\documentclass[12pt]{amsart}
\usepackage{amsmath, amsthm, amscd, amssymb, amsfonts, latexsym}
\usepackage{fullpage}
\usepackage{bm}
\usepackage[all]{xy}
\usepackage{tikz}
\usetikzlibrary{intersections,calc}

\usepackage{fancybox}

\usepackage{hyperref}

\newcommand{\kommentar}[1]{}
\newcommand{\C}{\mathbb{C}}
\newcommand{\F}{\mathbb{F}}
\newcommand{\R}{\mathbb{R}}
\newcommand{\Q}{\mathbb{Q}}

\newcommand{\Z}{\mathbb{Z}}

\newtheorem{thm}{Theorem}

\newtheorem{prop}[thm]{Proposition}

\newtheorem{lem}[thm]{Lemma}

 \definecolor{ffccww}{rgb}{1.,0.8,0.4}
\definecolor{qqzzcc}{rgb}{0.,0.6,0.8}
\definecolor{ccqqqq}{rgb}{0.8,0.,0.}
\definecolor{uuuuuu}{rgb}{0.26666666666666666,0.26666666666666666,0.26666666666666666}
\definecolor{ffdxqq}{rgb}{1.,0.8431372549019608,0.}

\renewcommand{\pmod}[1]{\,(\mathrm{mod}\,#1)}

\title{Spherical Heron triangles and elliptic curves}

\author{Tinghao Huang}
\address{Sun Yat-sen University (Zhuhai Campus), Tangjawan, Zhuhai, Guangdong  510275, China}
\email{huangth8@mail2.sysu.edu.cn}

\author{Matilde Lal\'in}
\address{Universit\'e de Montr\'eal,
  Pavillon Andr\'e-Aisenstadt,
  D\'epartement de math\'ematiques et de statistique,
  CP 6128,
  succ.\ Centre-ville, 
	Montr\'eal, Qu\'ebec, H3C~3J7, Canada}
  \email{mlalin@dms.umontreal.ca} 
\author{Olivier Mila}
\address{Centre de recherches math\'ematiques,
	Universit\'e de Montr\'eal,
  Pavillon Andr\'e-Aisen\-stadt,
  2920 Chemin de la tour,
	Montr\'eal, Qu\'ebec, H3T~1J4, Canada}
  \email{olivier.mila@umontreal.ca}

\thanks{This work is supported by the Swiss National Science Foundation, Project number \texttt{P2BEP2\_188144}, by the Natural Sciences and Engineering Research Council of Canada, Discovery Grant \texttt{355412-2013}, by the Fonds de recherche du Qu\'ebec - Nature et technologies, Projet de recherche en \'equipe \texttt{256442} and \texttt{300951}, and by Mitacs - Globalink Research Internship}


\begin{document}

\maketitle

\begin{abstract}
 We define spherical Heron triangles (spherical triangles with ``rational'' side-lengths and angles) and parametrize them via rational points of certain families of elliptic curves. We show that the congruent number problem has  infinitely many solutions for most areas in 
    the spherical setting and we find a spherical Heron triangle with rational medians. 
    We also explore the question of spherical triangles with a single rational median or a single
    a rational area bisector (median splitting the triangle in half), and discuss various problems involving isosceles spherical triangles.
\end{abstract}

\section{Introduction}

Problem D21 in Guy's book \cite{Guy} asks whether there are any triangles whose area is rational and  whose sides and medians  have rational lengths.  The question of rational medians was already considered by Euler \cite{Euler}, who parametrized triangles whose medians are rational, but without imposing the other conditions. To this day Guy's  problem D21 remains open. Buchholz and Rathbun \cite{Buchholz-Rathbun, Buchholz-Rathbun2} parametrized families with two rational medians using elliptic curves. Other authors have worked with the elliptic curves that appear from this problem \cite{Dujella-Peral,Dujella-Peral2, Ismail}. More generally, Heron triangles (triangles with rational side lengths and rational area) have been extensively studied by various authors \cite{Sastry, Kramer, Goins, Bremner-Heron, vanLuijk, Luca, Stanica, Beardon, Halbeisen}. More general cevians were studied in \cite{Buchholz-thesis,Laflamme-Lalin}.

In \cite{Hartshorne-vanLuijk} Hartshorne and van Luijk introduced the idea of studying rationality of lengths in hyperbolic triangles. The second and third named authors followed this idea and studied various problems related to finding rational cevians in hyperbolic triangles \cite{LalinMila}. It should be noted that a slightly different notion of rationality for hyperbolic triangles was considered by Brody and Schettler in \cite{Brody-Schettler}.

In this work we study some analogous problems for spherical triangles. To do this, we need to define the idea of rationality in this context. A spherical triangle is a triangle on the surface of the unit sphere whose sides are given by arcs in great circles, i.e., it is determined by the intersections of three planes passing through the center of the sphere with the surface of the sphere. We will focus on {\em proper} triangles, which satisfy that the sides $a, b, c$ and the angles $\alpha, \beta, \gamma$ are smaller than $\pi$. Thus, in a proper spherical triangle, we have 
\[\pi < \alpha + \beta+ \gamma <3\pi\]
and 
\[0<a +b+c<2\pi.\] 
The Gauss--Bonnet theorem implies that the area of such spherical triangle is given by 
\begin{equation}\label{eq:GB}
A=\alpha+\beta+\gamma -\pi.
\end{equation}
Following a convention analogous to what was adopted in \cite{Hartshorne-vanLuijk,LalinMila}, we will call an angle $\omega$, the area $A$, or a length $x$ rational  if and only if the sines and cosines of these quantities are rational, or equivalently, if  $e^{i\omega}, e^{iA},$ or  $e^{ix} \in \Q(i)$. In particular, notice that if $\alpha, \beta$, and $\gamma$ are rational, equation \eqref{eq:GB} implies that so is the area $A$. 

Recall that $e^{ix} \in \Q(i)$ if and only if 
$\cos(x) = \frac{1 - t^2}{1 + t^2}$ and $\sin(x) = \frac{2t}{1+t^2}$ for some $t \in \Q$. Indeed, we have 
\begin{equation}\label{eq:rational}
e^{ix}=\frac{i-t}{i+t}\in \Q(i) \Longleftrightarrow t=\frac{\sin(x)}{1+\cos(x)}\in \Q \Longleftrightarrow (\cos(x),\sin(x))=\left(\frac{1-t^2}{1+t^2},\frac{2t}{1+t^2}\right).
\end{equation}
By abuse of terminology we will call $t$ the \emph{rational side} (resp.\ \emph{rational angle}) of a 
spherical triangle if its side (resp.\ angle) is $x$.

In sum, a spherical triangle with area $A$, angles $\alpha, \beta, \gamma$ and sides $a, b, c$ is a {\it spherical Heron triangle} or {\it spherical rational triangle} if 
\[e^{ia}, e^{ib}, e^{ic}, e^{i\alpha}, e^{i\beta}, e^{i\gamma} \in \Q(i),\]
and this implies that $e^{iA} \in \Q(i)$ as well. 

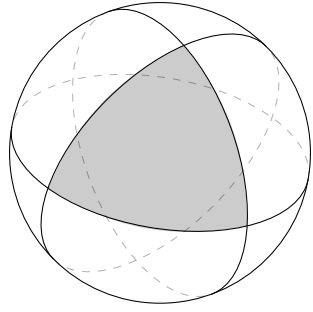
\begin{figure}
    \centering
\newcommand{\InterSec}[3]{%
  \path[name intersections={of=#1 and #2, by=#3, sort by=#1,total=\t}]
  \pgfextra{\xdef\InterNb{\t}}; }
   
       \begin{tikzpicture}

       \pgfmathsetmacro\R{2} 
       \draw (0,0,0) circle (\R); 

        \foreach \angle[count=\n from 1] in {-10,225,110} {

          \begin{scope}[rotate=\angle]
              \path[draw,dashed,opacity=0.3,name path global=d\n] (2,0) arc [start angle=0,
              end angle=180,
              x radius=2cm,
            y radius=1cm] ;
            \path[draw,name path global=s\n] (-2,0) arc [start angle=180,
              end angle=360,
              x radius=2cm,
            y radius=1cm] ;
          \end{scope}
        }

        \InterSec{s1}{s2}{I3} ;
        \InterSec{s1}{s3}{I2} ;
        \InterSec{s3}{s2}{I1} ;
        \fill[fill=black,opacity=0.2] (I1) to [bend left=19]  (I2) to [bend left=21] 
        (I3) to [bend left=20] (I1);

        \InterSec{d1}{d2}{J3} ;
        \InterSec{d1}{d3}{J2} ;
        \InterSec{d3}{d2}{J1} ;
      \end{tikzpicture}

    \caption{A spherical triangle.}
    \label{fig:my_label2}
\end{figure}
One of the goals of this article is to compare the situation in the spherical and hyperbolic worlds. 
In this sense, some of our results will be analogous to the ones in \cite{LalinMila}.

First we treat the generation of spherical Heron triangles. If we fix two sides, we obtain the following result. 
\begin{thm} \label{th:sides}
For all but finitely many choices of rational sides with parameters $u$ and $v$ there are 
infinitely many  spherical triangles  such that the third side and the angles are rational.
\end{thm}
This result is completely analogous to \cite[Theorem 3]{LalinMila}. It is achieved by parametrizing such triangles with points in the elliptic curve 
\[y^2 = x  (x -  (v + v^{-1})^2)  (x - (w + w^{-1})^2)\]
and showing  that for most values of $v,w\in \Q$, this elliptic curve has positive rank. 

Another approach, which follows naturally from extending the congruent number problem and the techniques of \cite{Goins}, is to fix an angle and the area. While this was achieved for the hyperbolic case in \cite[Theorem 1, Corollary 2]{LalinMila}, we encounter a difficulty here, as we are not able to construct the corresponding elliptic curve over $\Q$. Instead, we consider some particular cases. For the spherical congruent number problem, we obtain 
\begin{thm}\label{thm:congruent}
For all rational areas $m \neq 1$ there are 
infinitely many area $m$ right spherical triangles with rational angles and sides. Thus, the spherical congruent number problem has a positive solution. 
\end{thm}
This is achieved by working with the elliptic curve 
\[y^2 = x (x -2  m  (m^2 + 1)) (x - 4  m  (m^2 + 1)).\]
It is also known that the congruent number problem has a positive solution in the hyperbolic space \cite{LalinMila}. Thus, the Euclidean plane is very special from this point of view. 

We also consider the case of a isosceles triangle in this context, and likewise obtain infinitely many Heron isosceles triangles with prescribed area and repeated angle, for most choices of the parameters. 

A surprising result in the spherical setting is that problem D21 has a positive solution.
\begin{thm} \label{thm:equilateral}
There exists a unique rational equilateral spherical Heron triangle whose sides have lengths $\frac{\pi}{2}$ and whose angles measure  $\frac{\pi}{2}$. The medians of this triangle measure $\frac{\pi}{2}$ and are, therefore, rational.
\end{thm}   
This is contrary to the Euclidean plane and hyperbolic settings, where such equilateral triangle do not exist. More precisely, this is the first setting in which a positive solution can been found for the problem D21. 

We also explore and find positive results for the existence of triangles with rational sides and one rational median, isosceles triangles with rational sides and two rational medians, and certain existence results involving a rational area bisector. The results are analogous to what is known for the hyperbolic case.

Finally we embark on a detailed study of isosceles triangles with meridians and equators as sides. For these particular triangles, one has a guaranteed rational median/bisector/height, and the goal is to find that the other two cevians are rational. We obtain a positive result with infinitely many solutions for the heights, while the medians and bisectors reduce to only one solution given by the equilateral triangle from the D21 problem. We also consider the area bisector and obtain a negative result in this case. The problem in this case depends on a non-trivial argument (originally due to Flynn and Wetherell \cite{FlynnWetherell}) for finding all the rational points of a bielliptic curve of genus 2.   

The main geometric tools we will use are the following basic results of spherical trigonometry. A basic reference is \cite{Todhunter}. Consider a spherical triangle with area $A$, angles $\alpha, \beta,$ and $\gamma$ and side lengths $a,b,$ and $c$, where $a$ (resp.\ $b, c$) is opposite to $\alpha$ (resp. \ $\beta, \gamma$).

The spherical law of cosines says 
\begin{equation}\label{eq:cosines}
\cos(c)= \cos(a) \cos(b) + \sin(a) \sin(b) \cos(\gamma),
\end{equation}
and similarly for $\cos(a), \cos(b)$. 

The dual of \eqref{eq:cosines} is the supplemental law of cosines, which says 
\begin{equation}\label{eq:supplementalcosines}
\cos(\gamma)=-\cos(\alpha)\cos(\beta)+\sin(\alpha)\sin(\beta)\cos(c),
\end{equation}
and analogously for $\cos(\alpha), \cos(\beta)$. 

The spherical law of sines gives 
\begin{equation}\label{eq:sines}
\frac{\sin(\alpha)}{\sin(a)}=\frac{\sin(\beta)}{\sin(b)}=\frac{\sin(\gamma)}{\sin(c)}.
\end{equation}

This paper is organized as follows. 
Section~\ref{sec:angles} and Section~\ref{sec:sides} cover the 
parametrization of spherical Heron triangles in terms of angles and sides  respectively.
Section~\ref{sec:equilateral} is focused on medians in the simple case of equilateral 
triangles and includes the proof of Theorem~\ref{thm:equilateral}.
Section ~\ref{sec:medians} includes the parametrization of spherical triangles with rational side 
lengths and one rational median, while Section~\ref{sec:area-bisectors} is devoted to the dual computation of 
the parametrization of spherical Heron triangles with one rational area bisector. Isosceles triangles with meridians and equator as sides are considered in Section \ref{sec:isosceles}. We close the paper with Section \ref{sec:further}, where we discuss variations of the definition of rationality that could lead to future directions of research.

\section{Spherical Heron triangles - Angle parametrization}\label{sec:angles}

In this section we give a parametrization of spherical Heron triangles in terms of angles and area. We consider a triangle with angles $\alpha, \beta, \gamma \in (0,\pi)$ that are \emph{rational} (as defined in the introduction). Since the area is given by equation \eqref{eq:GB}, it is also rational. 

The supplemental spherical law of cosines \eqref{eq:supplementalcosines} implies that the cosines of the sides are also rational, and it remains to check that the sines of the sides are rational.

The spherical law of sines \eqref{eq:sines} implies that 
\[\sin(a)\sin(\beta)\sin(\gamma)=\sin(b)\sin(\alpha)\sin(\gamma)=\sin(c)\sin(\alpha)\sin(\beta).\]
We call this common quantity $\Delta_1$; observe that it is rational if and only if the sines of all the sides are rational.
Squaring the supplemental  spherical law of cosines \eqref{eq:supplementalcosines}, we get 
\[\sin^2(\alpha)\sin^2(\beta)(1-\sin^2(c))=(\cos(\gamma)+\cos(\alpha)\cos(\beta))^2.\]
This leads to 
\begin{equation}\label{eq:delta1}\Delta_1^2= \sin^2(\alpha)\sin^2(\beta)-(\cos(\gamma)+\cos(\alpha)\cos(\beta))^2 \in \Q.
\end{equation}
We remark that this expression is very similar to \cite[Eq. (6)]{LalinMila}, except that there is a sign difference on the right-hand side. From this point we can follow the treatment from \cite{LalinMila}. 
Using trigonometric identities, we can rewrite this as a symmetric expression in $\alpha, \beta, \gamma$:
\begin{align*}2\Delta_1^2=& -\cos(-\alpha+\beta+\gamma)-\cos(\alpha-\beta+\gamma)-\cos(\alpha+\beta-\gamma)\\&-\cos(\alpha+\beta+\gamma)-\cos(2\alpha)-\cos(2\beta)-\cos(2\gamma)-1.
\end{align*}
Substituting for $\gamma=A+\pi-\alpha-\beta$, expanding the cosines, and writing $c_A=\cos(A), s_A=\sin(A)$, etc, we have
\begin{align}\label{eq:c_Aandfriends}
      2 \Delta_1^2 =& 
  -(c_A^2 - s_A^2)\big[(c_\alpha c_\beta - s_\alpha s_\beta)^2 -  (c_\alpha s_\beta + c_\beta s_\alpha)^2\big]
                \\ \nonumber & 
  - 4 c_A s_A\big[c_\alpha s_\alpha (c_\beta^2 - s_\beta^2) + c_\beta s_\beta(c_\alpha^2 - s_\alpha^2)\big]
                \\ \nonumber & 
  + c_A\big[(c_\alpha c_\beta - s_\alpha s_\beta)^2 -  (c_\alpha s_\beta + c_\beta s_\alpha)^2
      + 2c_\alpha^2 + 2c_\beta^2 - 1\big]
                 \\ \nonumber & +
  4 s_A(c_\alpha s_\alpha c_\beta^2 + c_\beta s_\beta c_\alpha^2) - 2 c_\alpha^2 - 2 c_\beta^2 + 1.
\end{align}

Since we wish to express this in terms of rational angles, we set
\[t=\frac{\sin(\alpha)}{1+\cos(\alpha)}, \quad u=\frac{\sin(\beta)}{1+\cos(\beta)},  \quad m=\frac{\sin(A)}{1+\cos(A)},\]
and $w = (m^2+1)(u^2 + 1)(t^2 + 1)\Delta_1$,  equation \eqref{eq:c_Aandfriends} rewrites as:
\begin{align} \label{eq:w-t-angles}
    w^2 = & -4 m  (mu^2 - m + 2u)  (mt^2  +2t - m) \big[(mu^2 - m + 2u)t^2 
                 \\& + (-4mu + 2u^2 - 2)t - mu^2 + m - 2u\big]. \nonumber
\end{align} 
Here we differ from the situation of \cite[Eq. 8]{LalinMila}, where we were able to find a change of variables $\{t,w\}\rightarrow \{x,y\}$ turning the equation into a Weierstrass form. In this case, having the opposite sign on the right-hand side of \eqref{eq:w-t-angles} creates an obstruction to find a general solution to the equation that is defined over $\Q(u,m)$. By twisting $w$ by $i$, one can actually recover the change of variables leading to \cite[Eq. 9]{LalinMila}, which in this case it will not be defined over $\Q(u,m)$, but over over $\Q(i)(u,m)$. In \cite[Lemma 2.1]{LalinMila}, a point of infinite order $P$ over $\Q(u,m)$ was found, but this will only lead to a point over $\Q(i)(u,m)$ after twisting. In our case, we will not be able to conclude that our problem over the spherical side has infinite solutions. 

Instead, we proceed to examine two particular cases of interest: $u=1$ (a right triangle) and $u=t$ (an isosceles triangle).

\subsection{The case $u=1$} By setting $u=1$ equation \eqref{eq:w-t-angles} becomes 
\[w^2  =  -16 m (mt^2  +2t - m) (t^2 -2m t- 1)\]
 with a solution $(t,w)=(1,8m)$.

By applying the change of variables
\begin{align*}
y=&\frac{m((1+2m-m^2)(4mt^3-12mt-w)+(1-2m-m^2)(12mt^2-4m+wt))}{(t-1)^3},\\
x=& \frac{m(4(m^2t^2-(m-1)^2t+1)+w)}{(t-1)^2},
    \end{align*}
we obtain the Weierstrass form 
\begin{equation} \label{eq:congruent}
E_{u=1}: y^2 = x (x -2  m  (m^2 + 1)) (x - 4  m  (m^2 + 1)).
\end{equation}
We remark that \eqref{eq:congruent} appeared in \cite{LalinMila}. In fact, it was proven that $E(\C(m))$ is a $K3$-surface of rank 2, and that $P(m)=((m^2+1)(m+1)^2,(m^2+1)^2(m^2-1))$ and $Q(m)=(2m(m+1)^2,4im^2(m^2-1))$ are two independent points of infinite order. 

We claim that for every rational value of $m \notin \{-1,0,1\}$, the point 
$P(m)$ has infinite order on $E_{u=1}$.
Indeed, Mazur’s Theorem (see \cite{Mazur77, Mazur78}) implies that the torsion group of a rational elliptic curve has order at most 16. 
By looking at the points on $E_{u=1}$ of the form $\pm k P + \ell (0,0)$ for 
$k \in \{1,2,3,4\}, \ell \in \{0,1\}$, we see that we generically get 16 different points.
Thus for each value of $m$, either one of these points is non-torsion (from which it follows easily that $P(m)$ has infinite order), or they are all torsion.
In the latter case, together with $(0,0)$ we have 17 points, so two points of this list must coincide, and it is easily verified by looking at the equations for these points that this is only possible if $m \in \{-1,0,1\}$.

Finally, observe that that the conditions (e.g., sum of angles $> \pi$, etc.) for a set of parameters $(\alpha, \beta, \gamma, A)$ to give rise to an actual spherical triangle translate into \emph{open conditions} (i.e., involving strict inequalities) on the variables $t,u,m$, which in turn also translate into open conditions on the variables $x,y,m$.
Now by a theorem of Poincaré--Hurwitz (see \cite[Satz 11, p.\ 78]{Skolem}) 
the points $E_{u=1}(\Q)$ form a dense subset of $E_{u=1}(\R)$ as long as 
$E_{u=1}(\Q)$ is infinite and intersects both connected components of $E_{u=1}(\R)$.
Since the three torsion points having $y=0$ are rational (and lie across both connected components of $E_{u=1}(\R)$), and since unless $m\in\{-1,0,1\}$, the point $P(m)$ is a rational point of infinite order, we have proven the following:

\begin{thm}[Theorem~\ref{thm:congruent} in the introduction]
\label{th:thm4}
For every positive rational $m \neq 1$, the congruent number problem has a solution in the spherical context.
More precisely: for all rational areas $m \neq 1$ there are 
infinitely many area $m$ right spherical triangles with rational angles and sides.
\end{thm}
Note that this type of argument will be used later in the text to deduce existence of infinitely many triangles with given properties from elliptic curves having positive rank.

Observe that in the case $m=1$, the elliptic curve has rank zero, and using the change of variables (when it is defined) it is possible to show that the only possible solution is with $t=1$.
This corresponds to a triangle having area $\frac \pi 2$ (since $m=1$) and 2 angles also equal to $\frac \pi 2$ (since $u=t=1$).
Thus by \eqref{eq:GB}, the third angle is also $\frac \pi 2$, and the only rational triangle with area and one angle equal to $\frac \pi 2$ is the unique rational equilateral triangle, with all sides, angles and area equal to $\frac \pi 2$.

\subsection{The case $u=t$}  By setting $u=t$, and $w=w_1(mt^2+2t-m)$,  equation \eqref{eq:w-t-angles} becomes 
\[
    w_1^2=-4m(mt^4+4t^3-6mt^2-4t+m)
\]
with particular solution $(t,w_1) = (1,4m)$. We apply Cassels' algorithm \cite[p. 37]{Cassels}
and find the change of variables \begin{align*}
    y =& \frac{m}{2(t-1)^3}
     (-2m(m+1)t^3 +6m (m-1)t^2 + 6m (m+1)t \\
     &+2m(1-m)   +(mt+  m - t+ 1) w_1), \\
     x =& \frac{m (4mt - 2t^2 +2+ w_1)}{2(t-1)^2},
\end{align*}
that leads to the Weierstrass form 
\begin{equation}
E_{u=t}:y^2=x(x^2-m^2(1+m^2)).\end{equation}
\begin{lem}
The rank of the rational elliptic surface $E_{u=t}(\C(m))$ is 2. Its torsion group is isomorphic to $\Z/2\Z$. The points 
\[P(m)=\big(-m^2,m^2\big),\quad  Q(m)=\big( m(im-1) , (i +1)m^2(im-1)\big),\]
are generators of the free subgroup. 
\end{lem}
\begin{proof}
Note that $E_{u=t}$ is a rational elliptic surface with discriminant $\mathrm{disc}(E_{u=t}) = 64m^6(m^2+1)^3$. By Tate's algorithm \cite[IV.9]{Silverman-advanced} $E_{u=t}$ has  singularities at $m=0$ of type $I_0^*$, and $m=\pm i$ of type $III$. By the Shioda--Tate formula \cite[Corollary 6.7]{SS-book}, the rank of the N\'eron--Severi group is given by 
\begin{equation}\label{eq:ST}\rho(E)=\mathrm{rk}(E(\C(m))+2+\sum_{\nu} (m_\nu-1).\end{equation}
In our case, we obtain 
\[\rho(E_{u=t})=\mathrm{rk}(E_{u=t}(\C(m))+2+(5-1)+2\cdot(2-1)=\mathrm{rk}(E_{u=t}(\C(m))+8.\]
Since $\rho(E_{u=t})=10$ for rational elliptic surfaces, we conclude that $\mathrm{rk}(E_{u=t}(\C(m))=2$.

By \cite[Table 4.5]{MirandaPersson}, since the rank is $R=2$ and the Euler characteristic is $\chi=1$, we conclude  that the torsion is either  $\Z/2\Z$ or $\Z/2\Z\times \Z/2\Z$, but it is very clear that the only point of order 2 is $(0,0)$, and therefore the torsion is $\Z/2\Z$.

Now, if one only wants to show that $P(m)$ and $Q(m)$ are independent points of infinite order, the easiest way is to specialize at $m=1$ and verify (using Sage for instance) that the corresponding curve has rank 2 over $\Q(i)$ and admits those points as generators of the free part.
(In fact, checking that $P(m)$ is of infinite order is even easier: 
at $m=1$, we get the point $P=(-1,1)$ on the curve $y^2=x(x^2-2)$.
Since torsion injects into specialization and $2P=(\frac{9}{4}, -\frac{21}{8})$ has non-integral coordinates, one concludes that $P$ can not be torsion due to the Nagell--Lutz theorem.)
However, proving that they are actually generators is more involved. We can do this by computing the height pairing of the     Mordell--Weil group on the elliptic surface $E_{u=t}$. In order to do this we need to find the height pairing of both points. By formulas (6.14) and (6.15) in \cite{SS-book}, 
%
\begin{align}
\langle P, Q\rangle =& \chi +(P.O)+(Q.O)-(P.Q)-\sum_\nu \mathrm{contr}_\nu(P,Q), \label{eq:hPR}\\
h(P):=\langle P, P\rangle =& 2\chi +2(P.O)-\sum_\nu \mathrm{contr}_\nu(P), \label{eq:hP}
\end{align}
and similarly for $Q$. In the above formulas, $(P.Q)$ represents the intersection multiplicity of $P$ and $Q$ and $\mathrm{contr}_\nu(P,Q)$ represent certain correction terms given by the local contribution from the fiber at $\nu$ (see \cite[Definition 6.23]{SS-book}).

We look at \cite[Table 6.1]{SS-book}. For the singularity at $m=0$ of type $I_0^*$, we get $\mathrm{contr}_0=1$ unless the point intersects $\Theta_0$ in the fiber.  We have that $P(0)=Q(0)=(0,0)$ (the singular point), so they do not intersect $\Theta_0$  and therefore  $\mathrm{contr}_0(P)=\mathrm{contr}_0(Q)=1$. 
We also have that  
$\mathrm{contr}_0(P,Q)=1/2$ since they do not  intersect the same component.

For the singularities $\pm i$, of type  $III$, we have that $P(i)=P(-i)=(-1,1)\not = (0,0)$ (the singular point is again $(0,0)$) so we get $\mathrm{contr}_{\pm i}(P)=0$ since it intersects $\Theta_0$. We have that $Q(i)=(-2i,2+2i)\not = (0,0)$, so that $\mathrm{contr}_{i}(Q)=0$, but 
$Q(-i)= (0,0)$, so that $\mathrm{contr}_{-i}(Q)=1/2$. Finally we have $\mathrm{contr}_{\pm i}(P,Q)=0$ since they intersect different components.

We also have that $P\cdot O=Q\cdot O=0$, since the coordinates are polynomials, and  $P\cdot Q=0$ since the points do not intersect the same component at $(0,0)$, which is the only possible point where $P=Q$.

Since $\chi=1$, we obtain from \eqref{eq:hP} that $h(P)=2\cdot 1 +2\cdot 0 -1-2\cdot0=1$,  
$h(Q)=2\cdot 1 +2\cdot 0 -1-0-1/2=1/2$ and from \eqref{eq:hPR}, $\langle P, Q\rangle =1+0+0-0-1/2-2\cdot 0=1/2$.

On the one hand, we can compute the determinant of the Gram matrix associated to the height pairing of $P$ and $Q$. This gives
\begin{equation}\label{eq:disc}\left|\begin{array}{cc}
    1 & 1/2 \\
    1/2 & 1/2
\end{array}\right|=\frac{1}{4}.
\end{equation}

On the other hand, by the Determinant formula \cite[Corollary 6.39]{SS-book}), we have 
\begin{equation}\label{eq:det-formula}
|\mathrm{disc}\, \mathrm{NS}(E_{u=t})| =\frac{|\mathrm{disc}\, \mathrm{Triv}(E_{u=t}) \cdot \mathrm{disc}\, \mathrm{MWL}(E_{u=t})|}{|E_{u=t}(\C(v))_\mathrm{tor}|^2},
\end{equation}
where $\mathrm{MWL}(E_{u=t})$ is the Mordell--Weil lattice and $ \mathrm{Triv}(E_{u=t})$ is the trivial lattice. 

By \cite[Definition 7.3]{Shioda-MW},
\begin{equation}\label{eq:shiodaMW}
\mathrm{disc}\, \mathrm{Triv}(E_{u=t}) = \prod_{\nu}  m_\nu^{(1)},
\end{equation}
where $m_\nu^{(1)}$ is the number of simple components of the corresponding singular fiber. We have
$m_\nu^{(1)}=2$ if $\nu$ is of type $III$ and $m_\nu^{(1)}=4$ if $\nu$ is of type $I_0^*$. We thus get
\[\mathrm{disc}\, \mathrm{Triv}(E_{u=t}) = 16.\]

Since $\mathrm{disc}\, \mathrm{NS}(E_{u=t})=-1$ (as the N\'eron--Severi lattice of a rational elliptic surface is unimodular) and  $|E_{u=t}(\C(m))_\mathrm{tor}|=2$, equation \eqref{eq:det-formula} becomes
\begin{equation}\label{eq:discMWL}
|\mathrm{disc}\, \mathrm{MWL}(E_{u=t})|=\frac{1}{4}.
\end{equation}
In conclusion, we have obtained the same value as in \eqref{eq:disc}, this proves that the points $P,Q$ are generators for the free part of $E_{u=t}(\C(m))$.

\end{proof}

Using arguments similar to those in the proof of Theorem~\ref{th:thm4}, one gets:
\begin{thm}[Theorem~\ref{th:sides} in the introduction]
For all but finitely many combinations of  rational area $m$ and rational angle $u$ there are 
infinitely many  isosceles spherical triangles with area $m$ and the repeated angle $u$ such that the third angle and the sides are rational.
\end{thm}

\section{Spherical Heron triangles - Side length parametrization} \label{sec:sides}
In this section we parametrize spherical Heron triangles given by their side lengths. Let $a,b,c$ denote the side lengths of a spherical triangle, and assume that they are rational (as defined in the introduction, i.e., $e^{ia}, e^{ib}, e^{ic} \in \Q(i)$).
Let $\alpha$ (resp.\ $\beta, \gamma$) be the angles opposing the sides of length $a$ (resp.\ $b,c$).
By the spherical  law of cosines  \eqref{eq:cosines} the cosines of the angles are also rational, and it remains to check that the sines of the angles are rational. 

The spherical law of sines \eqref{eq:sines} implies that 
\[
    \sin(\alpha) \sin(b) \sin(c) = \sin(\beta) \sin(a) \sin(c) = \sin(\gamma) \sin(a) \sin(b).
\]
Call this quantity $\Delta_2$; it is rational if and only if the sines of all the angles are rational.

As in Section~\ref{sec:angles}, we square  the spherical law of cosines \eqref{eq:cosines} to get
\[
    \sin(a)^2 \sin(b)^2(1 - \sin(\gamma)^2) = ( \cos(a)\cos(b)- \cos(c))^2.
\]
Hence 
\begin{equation}\label{eq:delta2} 
    \Delta_2^2 = \sin(a)^2 \sin(b)^2 - (\cos(a) \cos(b)  - \cos(c))^2 \in \Q.
    \end{equation}

Applying the change of variables
\[u=\frac{\sin(a)}{1+\cos(a)}, \quad v=\frac{\sin(b)}{1+\cos(b)},  \quad w=\frac{\sin(c)}{1+\cos(c)},\]
we get the equation
\[
  D^2 = (-uvw + u + v + w)  (uvw - u + v + w)  (uvw + u - v + w)  (uvw + u + v - w),
\]
where $D = \frac 1 2 (u^2 + 1)(w^2 + 1) (v^2 + 1) \Delta_2$. This has a solution $(u,D) = (\frac{v+w}{1-vw}, 0)$. Applying \cite[p. 37]{Cassels}, we find the change of variables \begin{align*}
y=&\frac{(v+v^{-1})(w+w^{-1})(v+w)(vw-1)D}{vw(uvw-u+v+w)^2},\\
x=&-\frac{(v+v^{-1})(w+w^{-1})(uvw-u-v-w)}{uvw-u+v+w},
\end{align*}
that yields the Weierstrass form 
\begin{equation} \label{eq:Weierstrass-sides}
E_{v,w}: y^2 = x  (x -  (v + v^{-1})^2)  (x - (w + w^{-1})^2).
\end{equation}
We remark the similarly of \eqref{eq:Weierstrass-sides} with \cite[Eq. 12]{LalinMila}. Indeed, both curves are isomorphic, we can go from one to the other by the change $(x,y)\rightarrow (-x,iy)$, $(v,w)\rightarrow (iv,iw)$. Applying this to \cite[Lemma 3.1]{LalinMila} we immediately obtain the following result. 
\begin{lem}\label{lem:rank-computations-sides}
Let $E_v$ denote the $K3$-surface over $\C(w)$ resulting from fixing the parameter $v$. Its rank satisfies
    \[
     1 \leq \quad \mathrm{rk}(E_v(\C(w)))\quad \leq 2.
    \]
    In addition, the torsion group of $E_v$ is isomorphic to $\Z/4\Z\times\Z/2\Z$, generated by 
\[S_0(v,w)=\big((v+v^{-1})(w+w^{-1}),i(v+v^{-1})(w+w^{-1})(v^{-1}-w^{-1})(vw-1)\big)\]
and  \[S_1(v,w)=\big((v+v^{-1})^2, 0\big).\]

Finally, the point 
    \[
      R(v,w) = \big(vw(v+v^{-1})(w+w^{-1}),(v+v^{-1})(w+w^{-1})(v^2w^2-1)\big)
    \]  
   has infinite order on $E$.
    \end{lem}

Finally, an argument as in Theorem~\ref{th:thm4} gives:
\begin{thm}
For all but finitely many choices of rational sides with parameters $u$ and $v$ there are 
infinitely many  spherical triangles  such that the third side and the angles are rational.
\end{thm}

\section{Equilateral triangles} \label{sec:equilateral}
The goal of this section is to explore the existence of equilateral spherical Heron triangles. In fact, we prove:. 
\begin{thm}[Theorem~\ref{thm:equilateral} in the introduction]
There exists a unique rational equilateral spherical Heron triangle given by $a=b=c=\frac{\pi}{2}$ and $\alpha=\beta=\gamma=\frac{\pi}{2}$.
\end{thm}
We remark that for the triangle described in Theorem \ref{thm:equilateral} the medians have the same lengths as the sides and thus they are rational. Therefore, this provides a {\em positive answer to the problem D21 in the spherical world}. 

\begin{proof} For this we go back to equation \eqref{eq:delta1}, where we set $\alpha=\beta=\gamma$:
\[\Delta_1^2=1-3\cos^2(\alpha)-2\cos^3(\alpha)=(1-2\cos(\alpha))(\cos(\alpha)+1)^2.\]
Setting $u=\frac{\Delta_1}{\cos(\alpha)+1}$, the above equation can be rewritten as $u^2=1-2\cos(\alpha)$. Thus the solutions to the original equation are parametrized by \begin{equation}\label{eq:deltacos}\cos(\alpha)=\frac{1-u^2}{2} \mbox{ and } \Delta_1=\frac{u(3-u^2)}{2}.\end{equation}
Squaring the first equation of \eqref{eq:deltacos}, writing $4\cos(\alpha)^2=4-4\sin^2(\alpha)$, and setting $v=2\sin(\alpha)$, we obtain
\[v^2=-u^4+2u^2+3.\]
Making the change of variables  
\begin{equation*}
y= \frac{2v+3-u^3+u^2+u}{(u-1)^3},\qquad 
x= \frac{v+2}{(u-1)^2},
\end{equation*}
we get 
\[E: y^2=x(x^2-x+1),\]
and this curve has rank 0. It is not hard to see that \[E(\Q)=\{O,(0,0),(1,\pm 1)\}\cong \Z/4\Z.\] 
These points only yield to solutions of the form $\sin(\alpha)=\pm 1$, $\cos(\alpha)=0$, thus leading to  a triangle whose angles and sides are all equal to $\frac{\pi}{2}$. 
\end{proof}

If we relax the condition that the angles be rational, we find another surprising result. 

\begin{prop}
The only  equilateral triangle that has rational sides and  rational medians is the one that satisfies $a=b=c=\frac{\pi}{2}$ and $\alpha=\beta=\gamma=\frac{\pi}{2}$.
\end{prop}
\begin{proof}
  Consider an equilateral spherical triangle of side lengths $a$ and angles $\alpha$.
  Let $m$ denote the length of the median, and consider the half triangle defined by one median.
  This triangle has angles $\alpha, \frac{\alpha}{2}, \frac \pi 2$ and sides $a, \frac a 2, m$.

  Assume the length $a$ is rational, i.e., that $e^{ia} \in \Q(i)$.
  By the Pythagorean theorem (a particular case of the law of cosines \eqref{eq:cosines}), \begin{equation}\label{eq:Pyt}\cos(m) \cos\left(\frac a 2\right) = \cos(a).\end{equation} We immediately see that the above equation has solutions when $a=\pi, \frac{\pi}{2}$, when both sides of \eqref{eq:Pyt} equal zero. However, notice that $a=\pi$ is not a valid solution. Otherwise, we remark that 
$\cos(m) \in \Q$ if and only 
  if $p = \cos(\frac a 2) \in \Q$.
  Let $t = \sin(m)$.
  Squaring \eqref{eq:Pyt}, we get the following equation for $t$:
  \[
      (1 - t^2) p^2 = (2 p^2 -1)^2 \quad \text{i.e.} \quad s^2 = -4 p^4 + 5 p^2 - 1, 
  \]
  writing $s = pt$.
  Changing variables 
 \[ s=\frac{6y}{x^2},\quad p=\frac{x-6}{x}, \quad y=\frac{6s}{(p-1)^2},\quad x=-\frac{6}{p-1},\]
  we get the following elliptic curve: 
  \[
    y^2 = x^3-19x^2+96x-144=(x-3)(x-4)(x-12).
  \]
This elliptic curve has rank $0$ and
\[
   E(\Q)=\{O, (3,0), (4,0), (12,0)\} \cong \Z/2\Z\times \Z/2\Z.
 \]
 Taking $x=3,4,12$ gives $p=-1,-\frac{1}{2}, \frac{1}{2}$, and $a=\pi, \frac{2\pi}{3}, \frac{\pi}{3}$ respectively. Since $0<a<\frac{2\pi}{3}$, we can only take $a=\frac{\pi}{3}$. However, this gives \[\cos(m)=\frac{1}{2\cos\left(\frac{\pi}{6}\right)},\]
 which is irrational. 
\end{proof}

Similarly, relaxing the condition that the sides be rational, we get:
\begin{prop}
The only  equilateral triangle that has rational angles and  rational medians is the one that satisfies $a=b=c=\frac{\pi}{2}$ and $\alpha=\beta=\gamma=\frac{\pi}{2}$.
\end{prop}
\begin{proof}
The proof of this result proceeds in the same vein as the previous proposition. In this case the starting point in the Pythagorean theorem as a particular case of the supplementary law of cosines \eqref{eq:supplementalcosines}:
\begin{equation} \label{eq:Pyt2}
\cos(m) \sin\left(\frac{\alpha}{2}\right)=\cos(\alpha).    
\end{equation}
 We immediately see the solution $\alpha=\frac{\pi}{2}$ with $m=\frac{\pi}{2}$. Notice that in general $\frac{\pi}{3}<\alpha<\pi$, and therefore $\sin(\frac{\alpha}{2})\not = 0$. 
We remark that 
$\cos(m) \in \Q$ if and only 
  if $p = \sin(\frac a 2) \in \Q$.
  Let $t = \sin(m)$.
  Squaring \eqref{eq:Pyt2}, we get the following equation for $t$:
  \[
      (1 - t^2) p^2 = (1-2 p^2)^2,
  \]
  writing $s = pt$.
  This reduces to the same elliptic curve as in the previous result: 
  \[
    y^2 = x^3-19x^2+96x-144=(x-3)(x-4)(x-12).
  \]
Taking $x=3,4,12$ gives $p=-1,-\frac{1}{2}, \frac{1}{2}$, and the only values $\frac{\pi}{3}<\alpha<\pi$ are $\alpha= \frac{\pi}{6}, \frac{5\pi}{6}$. However, the sine function evaluated in these angles is not rational. 
\end{proof}

\section{Rational medians} 
\label{sec:medians}
The goal of this section is to study spherical triangles with one rational median.  We consider a spherical triangle with sides $a,b,c$ and opposite angles $\alpha, \beta, \gamma$ as before. Let $m$ denote the median at the angle $\alpha$, cutting the side $a$ into two equal parts. Denote by $\theta$ the angle at the intersection of $m$ and $a$ on the side of $\beta$ (the one on the side of $\gamma$ is $\pi-\theta$).  Applying the law of cosines \eqref{eq:cosines} to both triangles, we have \begin{align*}
\cos(b)=&\cos(m)\cos(a/2)+\sin(m)\sin(a/2)\cos(\pi-\theta),\\ 
\cos(c)=&\cos(m)\cos(a/2)+\sin(m)\sin(a/2)\cos(\theta). 
\end{align*}
Combining both equations, we obtain
\begin{equation} \label{eq:med}
2\cos(m) \cos(a/2)=\cos(b)+\cos(c).
\end{equation}
We assume that $a,b,c$ are rational, i.e., $e^{ia}, e^{ib}, e^{ic} \in \Q(i)$.  Then for $\cos(m)$ to be rational it is necessary and sufficient that $\cos(a/2)$ be rational. Since $a$ is already rational, this is equivalent to $a/2$ being rational. We need in addition that $\sin(m)$ be rational. For this, we square equation \eqref{eq:med} and obtain that \begin{equation}\label{eq:sinm}
4\cos^2(a/2)-(\cos(b)+\cos(c))^2= 4\sin^2(m)\cos^2(a/2).
\end{equation}
We remark that the right-hand side of \eqref{eq:sinm} should be the square of a rational number. 

Let 
\[w=\frac{\sin(a/2)}{1+\cos(a/2)}, \quad u=\frac{\sin(b)}{1+\cos(b)},  \quad v=\frac{\sin(c)}{1+\cos(c)}.\]
After simplification, we must solve
\[(1-w^2)^2(1+u^2)^2(1+v^2)^2-(1+w^2)^2(1-u^2v^2)^2=t^2.\]
By applying the change of variables \begin{align*}
y=&\frac{4(u^2+1)^2(w^2-1)}{(uv-1)^3}(2u^2v^3w^4+v^3w^4+u^5v^2w^4+3u^3v^2w^4+uv^2w^4+u^4vw^4+3u^2vw^4\\&+vw^4+u^5w^4+2u^3w^4-4u^4v^3w^2-4u^2v^3w^2-2v^3w^2-2u^5v^2w^2-2u^3v^2w^2-2uv^2w^2\\&-2u^4vw^2+tu^2vw^2-2u^2vw^2+tvw^2-2vw^2-2u^5w^2+tu^3w^2-4u^3w^2+tuw^2-4uw^2\\&+2u^2v^3+v^3+u^5v^2+3u^3v^2+uv^2+u^4v-tu^2v+3u^2v-tv+v+u^5-tu^3+2u^3\!-tu),\\
x=&\frac{2(u^2+1)^2(w^2-1)(u^2v^2w^2+v^2w^2+u^2w^2+w^2-u^2v^2-v^2-u^2+t-1)}{(uv-1)^2},
\end{align*}
we get the Weierstrass form 
\begin{align}\label{eq:weierstrassmedian}
E_{u,w}:y^2=&x(x^2-4(u^4w^4+3u^2w^4+w^4-2u^4w^2-2u^2w^2-2w^2+u^4+3u^2+1)x\nonumber\\&+4(u^2+1)^4(w-1)^2(w+1)^2(w^2+1)^2)
\end{align}
Thus, we obtain the following result. 
\begin{thm}
    A spherical triangle with rational side $b$ with parameter $u$ and rational half-side $a/2$ with parameter $w$  has a rational median 
    (intersecting the side $a$) if and only if it corresponds (using the above change of variables)
    to a rational point on the elliptic curve $E_{u,w}$.
\end{thm}
Again in this case we can be more specific about the arithmetic structure of $E_{u,w}$.
\begin{lem}\label{lem:rank-computations-medians}
Let $E_u$ (resp.\ $E_w$) denote the $K3$-surface over $\C(w)$ (resp.\ $\C(u)$) resulting from fixing the parameter $v$ (resp.\ $w$). 
    The rank of  $E_u(\C(w))$ satisfies
    \[
     2 \leq \quad \mathrm{rk}(E_u(\C(w)))\quad \leq 6,
    \]
    while the rank of $E_w(\C(u))$ satisfies 
    \[ 2 \leq \quad \mathrm{rk}(E_w(\C(u)))\quad \leq 4.
    \]
    In addition, the torsion group is isomorphic to $\Z/2\Z$, generated by $(0,0)$. 

Finally, the points 
    \[P(u,w)=\big((u^2+1)^2(w^2+1)^2,(u^2-1)(u^2+1)^2(w^2+1)^3\big)
    \]  
    and 
\[Q(u,w)=\big(4u^2(w^2+1)^2,4u(u^4-1)(w^2-1)(w^2+1)^2\big)\]
   have infinite order on $E_{u,w}$ and are independent.
    \end{lem}
\begin{proof}
First notice that the discriminant of $E_{u,w}$ is given by
\begin{align*}\mathrm{disc}=&4096(u^2+1)^8(w-1)^4(w+1)^4(w^2+1)^4(w^2-2u^2-1)(u^2w^2-u^2-2)\\ &\times (u^2w^2+2w^2-u^2)(2u^2w^2+w^2-1).
\end{align*}
First look at $E_u(\C(w))$. We have singularities at $w=\pm 1, \pm i$ of type $I_4$, and
$w=  \pm \sqrt{2u^2+1},$ 
$\pm\frac{\sqrt{u^2+2}}{u}, \pm \frac{u}{\sqrt{2+u^2}}, \pm\frac{1}{ \sqrt{2u^2+1}}$ of type $I_1$. Applying Shioda--Tate formula \eqref{eq:ST}, 
\[\rho(E_u)=\mathrm{rk}(E_{u}(\C(w))+2+4\cdot(4-1)=\mathrm{rk}(E_{u}(\C(w))+14,\]
and since $\rho(E_u)\leq 20$ for $K3$-surfaces, we can bound the rank by 6. 

For $E_w(\C(u))$, we have singularities at $u=\pm i$ of type $I_8$ as well as at the roots of the other polynomials of type $I_1$. Shioda--Tate formula \eqref{eq:ST} gives 
\[\rho(E_w)=\mathrm{rk}(E_{w}(\C(u))+2+2\cdot(8-1)=\mathrm{rk}(E_{w}(\C(u))+16,\]
and since $\rho(E_u)\leq 20$, we can bound the rank by 4.

The lower bound for the rank will follow from the fact that $P(u,w)$ and $Q(u,w)$ are independent points of infinite order. This can be deduced directly from specializing at $u=2$ and $v=2$. 
Indeed, for these values, we obtain the Weierstrass form $y^2=x(x^2-1300x+562500)$, $P=(625, 9375)$, $Q=(400, 9000)$. Notice that $2P=(\frac{3025}{36}, -\frac{1343375}{216})$ and $2Q=(\frac{648025}{1296}, -\frac{420552125}{46656})$, which have non-integral coordinates, showing that $P$ and $Q$ are of infinite order. Moreover,
the Mordell-Weil group has rank 2 with generators of the free part given by $A = (50, 5000)$ and $B = (1250, 25000)$, and one verifies directly that $P = A - B$ and $Q=2B$.
Thus these points are of infinite order and independent at this specialization, and since any relation of dependence or finite order would automatically descend to the specialization, we conclude that these points are also independent and of infinite order over $\C(u,v)$.

Finally, by \cite[Table 4.5]{MirandaPersson}, since the rank is $R\geq 2$ and the Euler characteristic is $\chi=1$, we conclude  that the torsion is either  $\Z/2\Z$ or $\Z/2\Z\times \Z/2\Z$, but one can immediately see that the only point of order 2 is $(0,0)$, and therefore the torsion is $\Z/2\Z$.
\end{proof}

\subsection{The case $a=b$}

Here we set $a=b$ in the previous discussion. The goal is to obtain two (equal) rational medians and three rational sides in an isosceles triangle. This is equivalent to imposing $u=\frac{2w}{1-w^2}$ in \eqref{eq:weierstrassmedian}.
\begin{align*}
E_{w}:y_0^2=&x_0\left(x_0^2
-\frac{4(w^2+1)^2(w^4+6w^2+1)}{(w^2-1)^2}x_0+\frac{4(w^2+1)^{10}}{(w^2-1)^6}\right).
\end{align*}
Making the change $x_0=\frac{(w^2+1)^2}{(w^2-1)^4}x$, $y_0=\frac{(w^2+1)^3}{(w^2-1)^6}y$, we obtain 
\begin{align}\label{eq:weierstrassmedianisosceles}
E_{w}:y^2=&x\left(x^2
-4(w^4+6w^2+1)(w^2-1)^2x+4(w^2+1)^{6}(w^2-1)^2\right).
\end{align}
Looking at the degree of the coefficients, we conclude that $\chi=4$. We find two points of infinite order
 \[P(w)=\big((w^2+1)^4,(w^2+1)^4(w^2-2w-1)(w^2+2w-1)\big)
    \]  
    and 
\[T(w)=\big(2(w-1)^2(w^2+1)^3,16w^2(w-1)^2(w^2+1)^3\big).\]    
    

One can check that $P$ and $T$ have infinite order by evaluating at $w=2$. This gives the curve $y^2=x(x^2-1476x+562500)$ and $P=(625, -4375)$, $T=(250, 8000)$. 
We have $2P=(\frac{75625}{196}, \frac{20301875}{2744})$ and $2T=(\frac{15625}{16}, -\frac{546875}{64})$, which have non-integral coordinates, showing that $P$ and $T$ are points of infinite order. Indeed, we find that the curve has rank 2 and a set of generators for the free part is given by $P$ and $T$.

As in Theorem~\ref{th:thm4}, we conclude:
\begin{thm}
For all but finitely many values of $w$, there are infinitely many isosceles triangles with rational sides, two of which correspond to $w$, and two rational (symmetric) medians.
\end{thm}

\section{Area bisector}\label{sec:area-bisectors}
This section considers the \emph{area bisector}, the geodesic segment from one vertex, meeting the opposite side and 
separating the triangle into two triangles of equal area. For the area bisector to be rational, we will demand that the length be rational, and also that the half-area of the triangle be rational. 

Consider a spherical triangle with sides $a,b,c$ having opposite angles 
$\alpha, \beta, \gamma$.
Let $m$ denote the  area bisector at angle $\alpha$, cutting $\alpha$ into $\alpha_1$ and $\alpha-\alpha_1$.
Denote by $\theta$ the angle at the intersection of $m$ and $a$, on the side of $\alpha_1$, and (assume) on the 
side of $\beta$.
Thus we have two triangles: one with angles $\alpha_1, \beta, \theta$ and one with 
$\alpha - \alpha_1, \gamma, \pi-\theta$.

By the supplemental law of cosines \eqref{eq:supplementalcosines} we have
\[
  \sin(\alpha_1) \sin(\beta)\cos(c) = \cos(\theta) + \cos(\alpha_1)\cos(\beta).
\]  
Combining this with the definition of area bisector 
\[
  2( \alpha_1 + \theta + \beta - \pi) = A \quad \text{i.e.}\quad
  \theta = \pi + \frac{A}{2} - \alpha_1 - \beta,
\]  
we get 
\begin{equation}\label{eq:acosthm}
  \sin(\alpha_1) \sin(\beta)\cos(c) = - \cos\left(\frac{A}{2} - \alpha_1 - \beta\right) + 
  \cos(\alpha_1)\cos(\beta).
\end{equation}

Using trigonometric identities, we get 
\[
    \tan(\alpha_1) = \frac{ \cos(\beta) -\cos(A/2)\cos(\beta) - \sin(A/2)\sin(\beta) }%
    {\cos(\beta)\sin(A/2) - \cos(A/2)\sin(\beta) + \cos(c)\sin(\beta)}.
\]
Using the supplemental law of cosines \eqref{eq:supplementalcosines} again:
\[
    \sin(\alpha) \sin(\beta)\cos(c)  = \cos(\gamma) + \cos(\alpha)\cos(\beta),
\]
we get 
\[
    \tan(\alpha_1) = \frac{(\cos(A/2) - 1)\cos(\beta)\sin(\alpha) + \sin(A/2)\sin(\alpha)\sin(\beta)}{\cos(A/2)\sin(\alpha)\sin(\beta) - (\sin(A/2)\sin(\alpha) + \cos(\alpha))\cos(\beta) - \cos(\gamma)}.
\]

Now using $\cos(\gamma) = -\cos(A - \alpha - \beta)$, and expanding the trigonometric identities we get 
\begin{align*}
    \tan(\alpha_1) =&
    -\big((\cos(A/2) - 1)\cos(\beta)\sin(\alpha) + \sin(A/2)\sin(\alpha)\sin(\beta)\big) \big/ \\
                   & \big((2\cos(\alpha)\sin(A/2)^2 - (2\cos(A/2) - 1)\sin(A/2)\sin(\alpha))\cos(\beta) \\
                   &- (2\cos(A/2)\cos(\alpha)\sin(A/2) + (2\sin(A/2)^2 + \cos(A/2) - 1)\sin(\alpha))\sin(\beta)\big).
\end{align*}

Hence the tangent of $\alpha_1$ is always rational if $\alpha, \beta$ and $A/2$ are (i.e.\ sines and cosines of 
these quantities).
Thus, for $\alpha_1$ to be a rational angle, we must ask that $\frac{1}{\cos(\alpha_1)} \in \Q$. Therefore, we need that
\[
    w^2 = 1 + \tan(\alpha_1)^2
\]
for some $w \in \Q$.
Applying the change of variables
\[n=\frac{\sin(A/2)}{1+\cos(A/2)}, \quad u=\frac{\sin(\beta)}{1+\cos(\beta)},  \quad t=\frac{\sin(\alpha)}{1+\cos(\alpha)},\]
and clearing a square (substituting $w = s^2w$), we get 
\begin{align*}
    w^2 = & 4 (n - u)^2  (nu + 1)^2 t^4  + 4  (n - u)  (nu + 1)  (-2n^3u + 3n^2u^2 - 3n^2 + 6nu - u^2 + 1) t^3  \\
    & + \big(n^6u^4 + 2n^6u^2 - 8n^5u^3 + 11n^4u^4 + n^6 + 8n^5u - 50n^4u^2 + 64n^3u^3  \\
    & - 13n^2u^4  + 11n^4 - 64n^3u + 86n^2u^2 - 24nu^3 + u^4 - 13n^2 + 24nu - 6u^2 + 1\big)t^2  \\
    & + 4  (-n + u)  (nu + 1)  (-2n^3u + 3n^2u^2 - 3n^2 + 6nu - u^2 + 1)t   + 4  (-n + u)^2  (nu + 1)^2,
\end{align*}
that has a rational point $(t,w)=(0, 2(-n+u)(nu+1))$.

We remark that this equation is the same as in \cite[Section 6]{LalinMila} after making the change of variables $u\rightarrow -u$, $t\rightarrow -t$. Thus we get
\begin{align}
    \label{eq:area-bisector}
E_{n,u}: y^2=&(x-(n^2+1)^2(nu^2+2u-n)^2)(x^2-(n^2+1)(n^4u^4-8n^2u^4-u^4+16n^3u^3\\
 \nonumber   & -16nu^3-6n^4u^2+32n^2u^2- 10u^2-16n^3u+16nu+n^4-8n^2-1)x\\
  \nonumber     & -(n^2+1)^2(nu^2+2u-n)^2(3n^2u^2-u^2-2n^3u+6nu-3n^2+1)^2).
 \end{align}
We, therefore, have the following result. 
 \begin{thm}
     A spherical Heron triangle with rational half-area with parameter $n$ and rational angle with parameter $u$ has one rational area bisector if and only if it corresponds  to a rational point of $E_{n,u}$.
 \end{thm}
 
The analogue of 
 \cite[Lemma 6.1]{LalinMila} gives us some information about the arithmetic structure of the $K3$-surface $E_n$, and in particular, that it has a point of infinite order. 
   \begin{lem} \label{lem:rank-computations-area-bisectors}
    The rank of the $K3$-surface $E_n$ satisfies 
    \[
        1 \leq \mathrm{rk}(E_n(\C(u))) \leq 4.
    \]
Moreover, $E_n$ has a torsion point of order 2 given by $\left((n^2+1)^2(nu^2+2u-n)^2,0\right)$.
    
The point 
    \[
    Q(n,u)=\Big( 0,(n^2+1)^2(nu^2+2u-n)^2(3n^2u^2-u^2-2n^3u+6nu-3n^2+1)\Big)
    \]
    is of infinite order. 
  \end{lem}

\section{Isosceles triangle with meridians and equator as sides} \label{sec:isosceles}
\begin{figure}[h]
  \includegraphics{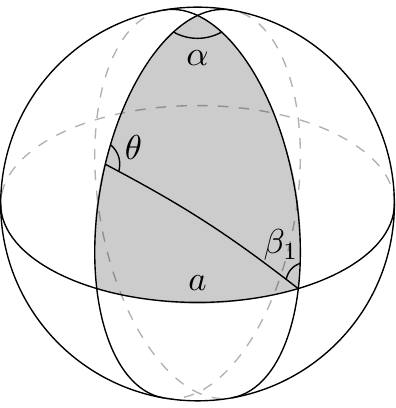}
    \caption{Schematic picture of the triangles under consideration.}
    \label{fig:my_label}
\end{figure}
      
In this section we consider a special family of spherical triangles. Namely we consider isosceles triangles with two half-meridians and a piece of the equator as sides. We will set that the side that is part of the equator has length $a$. The other two sides have length $\pi/2$. The angles are then $\alpha, \pi/2,\pi/2$. The median/bisector/height corresponding to $a$ is also $\pi/2$, and is, therefore, rational. 

Notice that the law of sines \eqref{eq:sines} gives $\sin(\alpha)=\sin(a)$ while the law of cosines \eqref{eq:cosines} gives $\cos(a)=\cos(\alpha)$.  Thus $a$ and $\alpha$ are rational simultaneously. Assume they are. 

Our goal is to study when the other two cevians are rational. Thus consider a cevian $d$ from $B$ to $b$, intersecting the side $b$ at angle $\theta$ on the side of the vertex $A$, and dividing the angle $\beta$ at $B$ into two angles  $\beta_1$ on the side of the vertex $A$ and $\pi/2-\beta_1$ on the side of $a$ (see Figure \ref{fig:my_label}).

\subsection{Median} If the other two cevians are medians of length $m$, then they divide the corresponding opposite side into two geodesics of length $\pi/4$. But the  law of cosines \eqref{eq:cosines} gives $\cos(m)=\cos(\pi/4)\cos(a)$. Since $\cos(\pi/4)$ is irrational, so is $m$, unless $\cos(a)=0$. But this is  only possible when $a=\alpha=\pi/2$, and this leads to  the equilateral triangle that appears as the sole solution of Theorem \ref{thm:equilateral}.

\subsection{Height} If the other two cevians are heights of length $h$, then $\theta=\pi/2$, and the triangle containing the sides $a$ and $h$ must be isosceles since it has two angles of $\pi/2$. Thus $h=a$ and any triangle with $a$ rational gives a solution. 

\subsection{Bisector} If the other two cevians are bisectors of length $\flat$, then $\beta_1=\pi/4$. By the law of sines \eqref{eq:sines}, $\frac{\sin(\theta)}{\sin(\pi/2)}=\frac{\sin(\alpha)}{\sin(\flat)}$. 
From this
\begin{equation}\label{eq:sinbisector}\sin(\theta)\sin(\flat)=\sin(\alpha).
\end{equation}
The supplemental law of cosines \eqref{eq:supplementalcosines} gives 
\begin{align*}\cos(\alpha)=&-\cos(\theta)\cos(\pi/4)+\sin(\theta)\sin(\pi/4)\cos(\flat),\\
\cos(\pi/2)=&-\cos(\pi-\theta)\cos(\pi/4)+\sin(\pi-\theta)\sin(\pi/4)\cos(\flat).
\end{align*}
Adding the above equations, \[ \sin(\theta)\cos(\flat)=\frac{\cos(\alpha)}{2\sin(\pi/4)}.\] 
By combining with equation \eqref{eq:sinbisector}
we obtain
\[\tan(\flat)=\tan(\alpha)2\sin(\pi/4).\]

Since $\tan(\alpha)$ is rational, and $\sin(\pi/4)$ is not, we must have $\tan(\alpha)=0,$ and therefore $\alpha=\pi/2$. This leads, once again, to  the equilateral triangle that appears as the sole solution of Theorem \ref{thm:equilateral}.

\subsection{Area bisector} If the other two cevians are area bisectors of length $v$, the areas of the half-triangles are $\alpha+\beta_1+\theta-\pi$ and $\pi-\beta_1-\theta$. Combining these two equations, \[\pi =\alpha/2+\theta+\beta_1.\]

By the supplemental law of cosines \eqref{eq:supplementalcosines},
\[\sin(\beta_1)\sin(\alpha)\cos(\pi/2)=\cos(\theta)+\cos(\beta_1) \cos(\alpha).\]
Writing $\cos(\theta)=-\cos(\alpha/2+\beta_1)$,
\[0=-\cos(\alpha/2)\cos(\beta_1)+\sin(\alpha/2)\sin(\beta_1)+\cos(\beta_1)\cos(\alpha).\]
This gives
\begin{equation}\label{eq:tanbeta}
\tan(\beta_1)=\frac{\cos(\alpha/2)-\cos(\alpha)}{\sin(\alpha/2)}.
\end{equation}
Since $\alpha+\theta+\beta_1-\pi$ is half the area of the triangle, it must be rational, and therefore, $\theta+\beta_1$ is rational, and since $\alpha/2+\theta+\beta_1=\pi$, we conclude that $\alpha/2$ is rational. 

By the law of sines \eqref{eq:sines} we have 
\[\frac{\sin(v)}{\sin(\alpha)}=\frac{\sin(\pi/2)}{\sin(\theta)}=\frac{1}{\sin(\theta)}.\]
Therefore, we need that $\sin(\theta)$ be rational.

By the supplemental law of cosines \eqref{eq:supplementalcosines}, 
\begin{align}
\cos(\alpha)=&-\cos(\theta)\cos(\beta_1)+\sin(\theta)\sin(\beta_1)\cos(v)\label{eq1}\\
\cos(\pi/2)=&-\cos(\pi-\theta)\cos(\pi/2-\beta_1)+\sin(\pi-\theta)\sin(\pi/2-\beta_1)\cos(v)\nonumber\\
=&\cos(\theta)\sin(\beta_1)+\sin(\theta)\cos(\beta_1)\cos(v).\label{eq2}
\end{align}
Multiplying \eqref{eq1} by $\sin(\beta_1)$, \eqref{eq2} by $\cos(\beta_1)$, and adding, we get 
\[\cos(\alpha) \sin(\beta_1)=\sin(\theta)\cos(v).\]
From this, we see that $\sin(\beta_1)$ must be rational. Since $\tan(\beta_1)$ must be rational by \eqref{eq:tanbeta}, then $\cos(\beta_1)$ is also rational. We have then 
\begin{equation}\label{eq:betaw}
\tan(\beta_1)^2+1=w^2
\end{equation}
Setting 
\[n=\frac{\sin(\alpha/2)}{1+\cos(\alpha/2)}\]
in \eqref{eq:tanbeta}, combining in \eqref{eq:betaw}, and substituting $w(n^2+1)\rightarrow w$,
we get
\[w^2=n^6 - 5n^4 + 11n^2 + 1.\]
We will need a lemma.
\begin{lem} \label{lem:gen2}
The only rational points on the genus 2 curve
\[C:Y^2=X^6-5X^4+11X^2+1\]
are $(0,\pm 1)$ and the two points at infinity. 
\end{lem}
Note that the points $X=0$ correspond to $n=0$ and yield a degenerate case with $\alpha=0$.
Thus, assuming the lemma, we see that there are no such triangles with rational area bisectors.

\begin{proof}[Proof of Lemma~\ref{lem:gen2}] We follow the method due to Flynn an Wetherell \cite{FlynnWetherell}. Notice that $C$ is a bielliptic curve of genus 2.  $C$ covers two elliptic curves: 
\[E^a:Y^2=x^3-5x^2+11x+1,\] 
\[E^b:Y^2=x^3+11x^2-5x+1,\]
with the maps $(X,Y)\rightarrow (X^2,Y)$ and $(X,Y)\rightarrow (1/X^2,Y/X^3)$. Both $E^a(\Q)=\langle (3,4)\rangle$ and $E^b(\Q)=\langle (-1,4)\rangle$ have rank 1, and $E^a\times E^b$ is isogenous to the Jacobian $J$ of $C$. Since the Jacobian has rank  2,  the more standard methods for finding rational points, such as Chabauty's theorem, cannot be applied.

Our goal is to apply Lemma 1.1(a) from \cite{FlynnWetherell}, to the curve $E^a$. Let \[F^a(x)=x^3-5x^2+11x+1.\] $F^a(x)$ is an irreducible polynomial over $\Q$. Let $\omega$ a root of $F^a(x)$.

First we do the 2-descent and find that  $E^a(\Q)/2E^a(\Q)=\{O,(3,4)\}$. Then \cite[Lemma 1.1(a)]{FlynnWetherell} asserts that if $(X,Y) \in C(\Q)$, then $x=X^2$ satisfies one of the following two equations.
\begin{align*}
E_1^a=& y^2=x(x^2+(\omega-5)x+\omega^2-5\omega+11),\\
E_2^a=& y^2=(3-\omega)x(x^2+(\omega-5)x+\omega^2-5\omega+11).
\end{align*}

We remark that $E_1^a$ has rank 0 and torsion isomorphic to $\Z/4\Z$ generated by 
\[\left(\frac{\omega^2}{4} - \frac{3\omega}{2} + \frac{13}{4}, \frac{\omega^2}{4} - \frac{3\omega}{2} + \frac{17}{4}\right).\]
Thus, the only affine points from $C(\Q)$ arising from $E_1^a$ are $(0,\pm 1)$. 

We now consider $E_2^a(\Q(\omega))$. A standard descent argument shows that the rank is 1, with two generators: $(0,0)$ of order 2 and 
\[P_0=\left(1,-\frac{\omega^2}{2}+3\omega-\frac{9}{2}\right)\]
of infinite order. We need to check that there are no extra points with rational $x$-coordinate. For this, we apply the argument from Section 2 in \cite{FlynnWetherell}, and reduce modulo $5$. (Remark that the prime $5$ satisfies the technical conditions required by \cite[Eq. (2.13)]{FlynnWetherell}.) Let us denote by $\;\widetilde{}\;$ the reduction modulo $5$. We see that $\widetilde{P_0}$ has order $28$ in $\widetilde{E_2^a}(\F_5(\widetilde{\omega}))$. Therefore, any point $P$ of $E_2^a(\Q(\omega))$ can be written uniquely as $P=S+nQ_0$, for $n \in \Z$, $Q_0=28P_0$ and $S$ a point in the set 
\[\{kP_0, kP_0+(0,0): k\in \Z,  -14< k<14 \}.\]
In the above set the only points that have rational $x$-coordinate when reduced to $\widetilde{E_2^a}(\F_5(\widetilde{\omega}))$  are those in
\[M:=\{O, (0,0), \pm P_0, \pm 10P_0, \pm 4P_0+(0,0), \pm 13P_0+(0,0)\}.\]
Of those, $O, (0,0)$, and $\pm P_0$ have actual rational $x$-coordinate when viewed in $E_2^a(\Q(\omega))$.

Next we will check that if a point of the form $S+nQ_0$ with $S\in M$ has rational $x$-coordinate, then necessarily $n=0$. 

We work modulo $5^5$ as in \cite[Example 3.1]{FlynnWetherell}. Eventually we want to compute the $x$ coordinate of $nQ_0$ for $n$ an arbitrary integer. To do this efficiently, it is convenient to work on the formal group of the elliptic curve. Thus, we
compute the  $z$-coordinate of $Q_0$, where  $z=-x/y$:
\[5(343 \omega^2 + 534 \omega +379) \pmod{5^5}.\]

In order to multiply by $n$, we will  combine the logarithm and the exponential. Therefore our next step is to find the $\log$ of the $z$-coordinate of $Q_0$ (\cite[Eq. (2.9)]{FlynnWetherell}):
\[5(18\omega^2 +534\omega + 429) \pmod{5^5}.\]
Now we substitute $n\log(z)$ into the exponential and find the $z$-coordinate of $nQ_0$ (\cite[Eq. (2.10)]{FlynnWetherell}):
\begin{equation} \label{eq:nQ}
5(18\omega^2+534\omega+429)n+ 5^3(18 \omega^2+5\omega+18)n^3+5^4(4\omega^2+4\omega+1)n^5 \pmod{5^5}.
\end{equation}
Finally we compute $1/x$ (\cite[Eq. (2.6)]{FlynnWetherell}):
\[5^2\cdot 49 n^2\omega^2+ 5^2\cdot 61 n^2 \omega +5^2(75n^4+97n^2) \pmod{5^5}.\]
In order to have a rational point of the form $nQ_0$, the coefficients of $\omega^2$ and $\omega$ must be $0$ in $\Z_5$. 
Thus, we must have $5^2\cdot 49 n^2=0$ in $\Z_5$. This has a double root at $n=0$, and  Strassman's Theorem implies that the total number of roots can not exceed 2. Hence, we conclude that $n=0$ is the only possible solution. 

One must then do the same procedure for $S+nQ_0$ for each of the elements $S\in M$. 

To work with  $(0,0)+nQ_0$, we replace the coordinates of  $(0,0)$ and the value of equation \eqref{eq:nQ} in \cite[Eq. (2.8)]{FlynnWetherell}. This gives  
\[5^2(97\omega^2+91+6\omega)n^2+ 5^4\cdot (3\omega^2+3)n^4\pmod{5^5}.\]
for the $z$-coordinate of $(0,0)+nQ_0$.
We compute $1/x$ to get
\[5^4\cdot 4 n^4\omega^2 + 5^4 n^4\omega + 5^4 \cdot 2n^4 \pmod{5^5}\]
and conclude that $n=0$ as before. 

For $P_0+nQ_0$, we obtain 
\begin{align*}
&1+5(231\omega^2+337 \omega+405)n
+5^2(116\omega^2+30\omega+104)n^2
+5^3(14\omega^2+21\omega+22)n^3
\\
&+5^4(4\omega^2+3\omega+3)n^4+5^4(\omega^2+4\omega+1)n^5 \pmod{5^5}
\end{align*}
for the $z$-coordinate, and
\begin{align*}&(5^4\cdot 3n^5 + 5^4\cdot 2n^4 + 5^3 \cdot 13n^3 + 5^2\cdot 71n^2 + 5\cdot 221n + 971)\omega^2 \\&+ (5^4 n^4 + 5^3\cdot 7n^3 + 5^2\cdot 124n^2 + 5\cdot 174n + 2028)\omega \\& + (5^4 n^5 + 5^4\cdot 4 n^4 + 5^3\cdot 8n^3 + 5^4n^2 + 5\cdot 197n + 2358)\pmod{5^5}
\end{align*}
for $1/x$.
Since $5\nmid 971$, the coefficient of $\omega^2$ cannot be $0$ in $\Z_5$.

For $10P_0+nQ_0$, we obtain 
\begin{align*}
&(2780\omega^2+ 1980\omega+ 1584)+5(546\omega^2+157 \omega+476 )n+5^2 (112\omega^2+88\omega +100)n^2 \\&+5^3(5\omega^2 +17\omega +8)n^3+5^4(\omega^2+2\omega)n^4 + 5^4(4\omega^2+2\omega+4)n^5 \pmod{5^5}
\end{align*}
for the $z$-coordinate, and 
\begin{align*}&(5^4 n^5 + 5^4\cdot 4 n^4 +5^3\cdot 13n^3 + 5^2 \cdot 42 n^2 + 5\cdot 551 n +  2971)\omega^2\\& + (5^4 \cdot 3 n^5+5^4 n^4 + 5^4\cdot 3n^3 + 5^4\cdot 3n^2 + 5\cdot 489n + 573)\omega\\&+ (5^4\cdot 4 n^5 + 5^4\cdot n^4 + 5^4\cdot 3 n^3 + 5^2\cdot 72n^2 + 5\cdot 503n + 2058)\pmod{5^5}\end{align*}
for $1/x$.
Since $5 \nmid 2971$, the coefficient of $\omega^2$ cannot be $0$ in $\Z_5$.

For $4P_0+(0,0)+nQ_0$, we have 
\begin{align*}
&(2740\omega^2+1325\omega+2769)+5(389\omega^2+558 \omega+499)n+5^2(12\omega^2+98\omega+40)n^2\\&+
5^3(20\omega^2+3\omega+17)n^3+5^4(\omega^2+2\omega)n^4+5^4(\omega^2+3\omega+1)n^5 \pmod{5^5}\end{align*}
for the $z$-coordinate, and 
\begin{align*}
&(5^4 \cdot 4 n^5 + 5^4 \cdot 4 n^4 +5^3\cdot 22 n^3+ 5^2\cdot 107n^2 + 5\cdot 259 n +2356)\omega^2 \\&+ (5^4\cdot 2 n^5 +5^4 n^4 + 5^4\cdot 4n^3 + 5\cdot 136 n + 2788)\omega\\& + (5^4 n^5 + 5^4 n^4 + 5^2\cdot 47 n^2 + 5\cdot 607n + 313)\pmod{5^5}
\end{align*}
for $1/x$.
Since $5\nmid 2356$, the coefficient of $\omega^2$ cannot be $0$ in $\Z_5$.

For $13P_0+(0,0)+nQ_0$, we have 
\begin{align*}
&(2585\omega^2+1595\omega+1951)+5(149\omega^2+388\omega+390)n+5^2(111\omega^2+110\omega+39)n^2\\&+5^3(21\omega^2+9\omega+13)n^3+5^4(4\omega^2+3\omega+3)n^4 +5^4(4\omega^2+\omega+4)n^5 \pmod{5^5}  \end{align*}
for the $z$-coordinate, and
\begin{align*}
&(5^4\cdot 2 n^5 + 5^4\cdot 2 n^4 + 5^3\cdot 12 n^3+ 5^2\cdot 36 n^2 + 5\cdot 359 n + 2456) \omega^2\\& + (5^4 n^4 + 5^4\cdot 3 n^3 + 5^2\cdot 84 n^2 + 5\cdot 596 n + 3118)\omega \\&+ (5^4\cdot 4 n^5 + 5^4\cdot 4 n^4 + 5^3\cdot 17 n^3 + 5^3\cdot 18 n^2 + 5\cdot 403n + 803)  \pmod{5^5},
\end{align*}
for $1/x$.
Since $5\nmid 2456$, the coefficient of $\omega^2$ cannot be $0$ in $\Z_5$.

Finally, remark that we do not have to consider the points of the form $-P_0+nQ_0$, $-10P_0+nQ_0$,  $-4P_0+(0,0)+nQ_0$, and $-13P_0+(0,0)+nQ_0$, separately since these points can be obtained by multiplying the previous cases by $-1$.

Thus, we conclude that $n=0$. We  examine the rational $x$-coordinates of the $S \in M$, and conclude that the only possibilities for points having rational $x$-coordinates are $0,1$ coming from $(0,0)$ and $\pm P_0$. It is immediate to see that $X=1$ does not lead to points in $C(\Q)$, and therefore the only possibly solution is $X=0$, leading to a degenerate triangle as discussed before. 
\end{proof}

\section{Further research}\label{sec:further}
There are many topics of further research based on this current work. First, one could try considering different versions of ``rationality'' for triangles. 
One natural way would be to relax the condition that all trigonometric functions of the sides and angles/area be rational, and to call a length/angle rational if, say, its tangent is rational (compare \cite{Goins} and \cite{LalinMila}).
Another way would be to call a spherical length rational if the length of the straight segment (inside the sphere) joining its two endpoints is rational. If $a$ denotes the spherical length, it is not hard to see that this corresponds to $\sin(a/2)$ being rational.

Yet another definition of rationality would be that length / angles be rational multiples of $\pi$.
It is easy to see that the isosceles triangle with apex on the north pole and bottom side on the equator of length $\frac p q \pi$ (see Figure~\ref{fig:my_label}) has all its sides, angles and area rational in this sense.
It would be interesting to know if there exist triangles having this property that do not come from this construction.

It is also interesting to consider the necessary conditions that prevent a spherical triangle from having multiple rational cevians (heights, medians, area bisectors, etc). This point of view is, to some extent, opposite to the investigation in Section \ref{sec:isosceles}. For example, one can prove that if a triangle is isosceles with the angle between the identical sides equal to $\frac{\pi}{2}$, then the medians cannot be rational simultaneously. It is natural then to wonder which of these assumptions can be lifted. 

Another possible direction for further research is the construction of high rank elliptic curves as in \cite{Dujella-Peral,IzadiNabardi}. More specifically, the authors of  \cite{IzadiNabardi} used Heron's formula to derive elliptic curves with high ranks. As there is an analog of Heron's formula in the spherical world, namely L'Huilier's formula, it would be interesting to try to construct elliptic curves with high ranks by adapting the method of \cite{IzadiNabardi}.

\bibliographystyle{amsalpha}
\bibliography{MatildeBibliography} 
 
 \end{document}